\documentclass[11pt]{article}
\usepackage{geometry,amsmath,amsthm,amssymb,latexsym}

\usepackage[svgnames]{xcolor}
\usepackage[colorlinks=true,linkcolor=MidnightBlue,citecolor=MidnightBlue,urlcolor=MidnightBlue,pdfpagelabels=false]{hyperref}
\usepackage{indentfirst}
\usepackage{thmtools}
\theoremstyle{plain}
  \declaretheorem[numberwithin=section]{theorem}
  \declaretheorem[numberlike=theorem]{corollary}
  \declaretheorem[numberlike=theorem]{proposition}
  \declaretheorem[numberlike=theorem]{lemma}
  \declaretheorem[numberlike=theorem]{conjecture}
  \declaretheorem[numberlike=theorem]{question}
\theoremstyle{definition}
  
  \declaretheorem[numberlike=theorem]{example}
  \declaretheorem[numberlike=theorem]{remark}
\newenvironment{acknowledgements}{\bigskip\textbf{Acknowledgements.}}{}

\newcommand{\N}{\mathbb{N}}

\newcommand{\multichoose}[2]{\left(\!\middle(\genfrac{}{}{0pt}{}{#1}{#2}\middle)\!\right)}

\begin{document}

\title{The unlabeled list color function of disconnected graphs}

\author{ 
Hemanshu Kaul
\thanks{Department of Applied Mathematics, Illinois Institute of Technology, Chicago, IL, USA (kaul@iit.edu)}
\and
Jeffrey A. Mudrock
\thanks{Department of Mathematics and Statistics, University of South Alabama, Mobile, AL, USA (mudrock@southalabama.edu)}
\and
Armin Straub
\thanks{Department of Mathematics and Statistics, University of South Alabama, Mobile, AL, USA (straub@southalabama.edu)}
}

\date{July 18, 2026}

\maketitle

\begin{abstract}
  Given a graph $G$, its chromatic polynomial $P (G, k)$ counts
  proper $k$-colorings, while the corresponding list color function $P_{\ell}
  (G, k)$ counts the minimum number of proper colorings across all assignments
  of $k$ colors to each vertex. While it is clear that $P_{\ell} (G, k) \leq
  P (G, k)$, Donner showed in 1992 that $P_{\ell} (G, k) = P (G, k)$ whenever $k$ is sufficiently large.  In 1985, Hanlon defined and studied the chromatic polynomial
  for an unlabeled graph. A list version of Hanlon's notion was introduced in
  2024 by Kaul and Mudrock, who further raised the question of whether the analog
  of Donner's result holds in the unlabeled case. While they proved this for
  all connected point-determining graphs, even the case of the edgeless graph on $n$
  vertices remained open and was posed as a conjecture. We prove
  this conjecture and show that it implies that, more generally, a
  disconnected graph satisfies the unlabeled analog of Donner's result if all
  of its connected components do.

\medskip


\end{abstract}

\section{Introduction} 
All graphs in what follows are nonempty, finite, simple graphs. Generally speaking, we follow West~\cite{W01} for terminology and notation. The set of natural numbers is $\N = \{1, 2, 3, \ldots\}$. For $m \in \N$, we write $[m]$ for the set $\{1, \ldots, m\}$.

In the classical vertex coloring problem, we wish to color the vertices of a graph $G$ with up to $k$ colors from $[k]$ so that adjacent vertices
receive different colors, a so-called \emph{proper $k$-coloring}. The
\emph{chromatic number} of a graph, denoted $\chi (G)$, is the smallest $k$
such that $G$ has a proper $k$-coloring. In 1912, Birkhoff~\cite{B12}
introduced the notion of the chromatic polynomial with the hope of using it to
make progress on the Four Color Problem. For $k \in \N$, the \emph{chromatic
polynomial} of a graph $G$, $P (G, k)$, is the number of proper $k$-colorings
of $G$. It is well-known that $P (G, k)$ is a monic polynomial in $k$ of
degree $|V (G) |$ (e.g., see~\cite{B94}).

We say that two proper $k$-colorings $f, g$ of $G$ are \emph{equivalent up to
labeling} if there is an automorphism $\pi$ of $G$ such that $f (v) = g (\pi
(v))$ for all vertices $v \in V (G)$. Let $P^u (G, k)$ denote the number of
resulting equivalence classes of proper $k$-colorings of $G$. The $u$ in the
superscript indicates that $P^u (G, k)$ is the unlabeled analog of the
chromatic polynomial $P (G, k)$. Indeed, if $\mathcal{G}$ is the unlabeled
graph corresponding to $G$, that is $\mathcal{G}$ is the isomorphism class of $G$, then $P^u (G, k)$ equals the chromatic polynomial
$P (\mathcal{G}, k)$ as introduced and studied by Hanlon~\cite{H85} in 1985.
To avoid possible confusion, we will operate on labeled graphs $G$ throughout
what follows.

List coloring is a well-known variation on classical vertex coloring which was
introduced independently by Vizing~\cite{V76} and Erd{\H o}s, Rubin, and
Taylor~\cite{ET79} in the 1970s. For list coloring, we associate a
\emph{list assignment} $L$ with a graph $G$ such that each vertex $v \in V
(G)$ is assigned a set of available colors $L (v)$. We say $G$ is
\emph{$L$-colorable} if there is a proper coloring $f$ of $G$ such that $f
(v) \in L (v)$ for each $v \in V (G)$. We refer to $f$ as a \emph{proper
$L$-coloring} of $G$. A list assignment $L$ for $G$ is called a
\emph{$k$-assignment} if $|L (v) | = k$ for each $v \in V (G)$. The
\emph{list chromatic number} of a graph $G$, denoted $\chi_{\ell} (G)$, is
the smallest $k$ such that $G$ is $L$-colorable whenever $L$ is a
$k$-assignment for $G$. It is immediately obvious that $\chi (G) \leq
\chi_{\ell} (G)$ for any graph $G$.

The notion of chromatic polynomial was extended to list coloring in the early
1990s by Kostochka and Sidorenko~\cite{KS90}. If $L$ is a list assignment
for $G$, $P (G, L)$ denotes the number of proper $L$-colorings of
$G$. The \emph{list color function} $P_{\ell} (G, k)$ is the minimum value
of $P (G, L)$ where the minimum is taken over all possible $k$-assignments $L$
for $G$. Since a $k$-assignment could assign the same $k$ colors to every
vertex of $G$, $P_{\ell} (G, k) \leq P (G, k)$ for each $k \in \N$. In
general, the list color function can differ significantly from the chromatic
polynomial for small values of $k$. However, in 1992, answering a question of
Kostochka and Sidorenko~\cite{KS90}, Donner~\cite{D92} showed the
following.

\begin{theorem}[Donner~\cite{D92}]
  \label{thm:donner}For any graph $G$ there is an $N \in \N$ such that
  $P_{\ell} (G, k) = P (G, k)$ whenever $k \geq N$.
\end{theorem}

The threshold $N$ was improved to $N \le |E(G)|-1$ in~\cite{DZ22} (improving upon bounds in~\cite{T09} and~\cite{WQ17}).

Following~\cite{KM25}, we now introduce an unlabeled version of list
coloring. For a list assignment $L$ for $G$, let $P^u (G, L)$ denote the
number of equivalence classes of proper $L$-colorings of $G$ where we again
consider two proper $L$-colorings $f, g$ equivalent up to labeling if $f = g \circ
\pi$ for some automorphism $\pi$ of $G$. Likewise, we define the \emph{unlabeled list color function}
$P^u_{\ell} (G, k)$ as the minimum value of $P^u (G, L)$ where the minimum is
taken over all possible $k$-assignments $L$ for $G$.

\begin{remark}
  If $\mathcal{G}$ is the unlabeled graph corresponding to $G$, then
  $P^u_{\ell} (G, k)$ equals the list color function $P_{\ell} (\mathcal{G},
  k)$ introduced by the first two authors in \cite{KM25}.
\end{remark}

Note that, if $L_k$ is the $k$-assignment for $G$ that assigns $[k]$ to each
vertex of $G$, then $P^u (G, L_k) = P^u (G, k)$. We thus clearly have
$P^u_{\ell} (G, k) \leq P^u (G, k)$ for each $k \in \N$. In light of
Theorem~\ref{thm:donner}, the following question, raised by the first two
authors in \cite{KM25}, is natural.

\begin{question}[Kaul, Mudrock \cite{KM25}]
  \label{ques: fundamental}For which graphs $G$ does there exist an $N \in \N$
  such that $P_{\ell}^u (G, k) = P^u (G, k)$ whenever $k \geq N$?
\end{question}

We suspect that the answer to Question~\ref{ques: fundamental} is all graphs.
It was verified in~\cite{KM25} that all point-determining graphs satisfy the statement of Question~\ref{ques: fundamental}.  It is known that almost all graphs are point-determining (see~\cite{ER63} and note that asymmetric graphs are point-determining).

It was further observed in~\cite{KM25} that, if $G_1$ and $G_2$ are
connected and non-isomorphic graphs, then
\begin{equation}
  P_{\ell}^u (G_1 + G_2, k) = P_{\ell}^u (G_1, k) P_{\ell}^u (G_2, k) .
  \label{eq:disjoint}
\end{equation}
Here, $G_1 + G_2$ denotes the disjoint union of $G_1$ and $G_2$. We note that
the relation~\eqref{eq:disjoint} continues to hold when $G_1$ and $G_2$ are
disconnected graphs such that no connected component of $G_1$ is isomorphic to
a connected component of $G_2$. On the other hand, disconnected graphs that
have two distinct, isomorphic, connected components (in which case
\eqref{eq:disjoint} does not hold in general) posed an interesting challenge
in~\cite{KM25}. In particular, the following provocatively simple case of
Question~\ref{ques: fundamental} remained open.

\begin{conjecture}[Kaul, Mudrock \cite{KM25}]
  \label{conj: provoke}Let $\bar{K}_n$ be an edgeless graph on $n$ vertices.
  For each $n \in \N$ there is an $N \in \N$ such that $P_{\ell}^u (\bar{K}_n,
  k) = P^u (\bar{K}_n, k)$ whenever $k \geq N$.
\end{conjecture}

The main result of this paper is to prove this conjecture by showing the
slightly stronger conclusion that
\begin{equation}
  P_{\ell}^u (\bar{K}_n, k) = P^u (\bar{K}_n, k) = \binom{k + n - 1}{n} \quad
  \text{for all $k, n \in \N$} . \label{eq:provoke}
\end{equation}
Let us fix $\bar{K}_n$ to be the edgeless graph with vertex set $[n]$. Note
that a $k$-assignment $L$ for $\bar{K}_n$ assigns to each $j \in [n]$ a
$k$-element set of colors $L (j)$. Since there are no edges, every tuple
$(c_1, c_2, \ldots, c_n)$ with $c_j \in L (j)$ represents a proper
$L$-coloring of $\bar{K}_n$. Such colorings are equivalent up to labeling if
and only if the corresponding tuples are a permutation of each other. The
equivalence class of the coloring represented by the tuple $(c_1, c_2, \ldots,
c_n)$ is therefore uniquely represented by the multiset $[c_1, c_2, \ldots,
c_n]$. Here, and below, we use square brackets $[\ldots]$ in place of $\{
\ldots \}$ to denote multisets and to avoid confusion with regular set
notation. Note that in the case where the sets $L (j)$ are all equal to
$[k]$, we obtain multisets $[c_1, c_2, \ldots, c_n]$ of $n$ elements drawn
from $k$ choices, the number of such multisets is well~known to be the binomial coefficient
$\binom{k + n - 1}{n}$. The claim~\eqref{eq:provoke} is therefore equivalent
to the following result (where we write $C_j$ in place of $L (j)$).

\begin{theorem}
  \label{thm:edgeless}Suppose that $C_1, \ldots, C_n$ are sets of size $k$.
  Then, the set of multisets
  \begin{equation*}
    \left\{ [c_1, \ldots, c_n] \; : \; c_j \in C_j \right\}
  \end{equation*}
  has size at least $\binom{k + n - 1}{n}$.
\end{theorem}

We prove Theorem~\ref{thm:edgeless} in Section~\ref{sec:edgeless}.  While working on extending this result, the first two authors were contacted by W.~T.~Gowers~\cite{G26} in May 2026. As part of an experiment, Gowers had asked OpenAI's ChatGPT 5.5 Pro to look for an open problem in combinatorics that it considered approachable and attempt to solve it. The model selected Conjecture~\ref{conj: provoke} and produced a possible affirmative proof which Gowers shared with the first two authors. We have confirmed the correctness of the AI-produced proof and, since it differs from our proof in Section~\ref{sec:edgeless}, we have included a paraphrased version of the proof in Appendix~\ref{sec:edgeless:ai}.

In addition to an alternative proof of Theorem~\ref{thm:edgeless}, the AI-produced output revealed the following applications of
Theorem~\ref{thm:edgeless} to Question~\ref{ques: fundamental}. At the time of Gowers' communication, the authors had only realized for themselves the special case of Corollary~\ref{cor:mono} when $G$ is a complete graph. In the following, given a graph $G$ and $m \in \N$, we write $m G$ for the disjoint union of $m$ copies of $G$. Similarly, given vertex disjoint graphs $G_1, \ldots, G_r$, we write $\sum_{j = 1}^r G_j$ for their disjoint union.

\begin{corollary}
  \label{cor:mono}Let $G$ be a connected graph such that there exists $N \in
  \N$ so that $P_{\ell}^u (G, k) = P^u (G, k)$ whenever $k \geq N$. Then,
  for any $m \in \N$, $P_{\ell}^u (m G, k) = P^u (m G, k)$ whenever $k
  \geq N$.
\end{corollary}

\begin{proof}
  Let $k \geq N$. Let $\mathcal{C}$ be the set of equivalence classes of
  proper $k$-colorings of $G$. It readily follows from unraveling the
  definitions that
  \begin{equation}
    P^u (m G, k) = \left| \left\{ [c_1, \ldots, c_m] \; : \; c_j \in
    \mathcal{C} \right\} \right| = \binom{| \mathcal{C} | + m - 1}{m} .
    \label{eq:PumGk}
  \end{equation}
  On the other hand, let $L^{(j)}$, for $j \in [m]$, be arbitrary
  $k$-assignments for $G$. Let $L$ be the corresponding $k$-assignment for $m
  G$. Further, let $\mathcal{C}_j$ be the set of equivalence classes of proper
  $L^{(j)}$-colorings of $G$. Since $k \geq N$, we have $P_{\ell}^u (G,
  k) = P^u (G, k)$ by assumption. Accordingly, $| \mathcal{C}_j | \geq |
  \mathcal{C} |$. Therefore, by Theorem~\ref{thm:edgeless} (whose statement
  obviously remains true if the sets $C_1, \ldots, C_n$ are allowed to have
  size $k$ or bigger),
  \begin{equation*}
    P^u (m G, L) = \left| \left\{ [c_1, \ldots, c_m] \; : \; c_j \in
     \mathcal{C}_j \right\} \right| \geq \binom{| \mathcal{C} | + m -
     1}{m} .
  \end{equation*}
  In combination with \eqref{eq:PumGk} this proves that $P_{\ell}^u (m G, k) =
  P^u (m G, k)$.
\end{proof}

\begin{corollary}
  \label{cor:multi}Let $G_1, \ldots, G_r$ be pairwise non-isomorphic connected
  graphs such that there exists $N \in \N$ so that $P_{\ell}^u (G_j, k) = P^u
  (G_j, k)$ for each $j \in [r]$ whenever $k \geq N$. Then, for any $m_1,
  \ldots, m_r \in \N$, the graph $G = \sum_{j = 1}^r m_j G_j$ satisfies
  $P_{\ell}^u (G, k) = P^u (G, k)$ whenever $k \geq N$.
\end{corollary}

\begin{proof}
  This follows from Corollary~\ref{cor:mono} combined with the
  relation~\eqref{eq:disjoint} which, suitably iterated, in the present
  situation implies that
  \begin{equation*}
    P_{\ell}^u (G, k) = \prod_{j = 1}^r P_{\ell}^u (m_j G_j, k) .
  \end{equation*}
\end{proof}

Notice that Corollary~\ref{cor:multi} implies that the answer to Question~\ref{ques: fundamental} is yes if for every connected graph $G$ there exists an $N \in \N$ such that $P_{\ell}^u(G,k) = P^u(G,k)$ whenever $k \geq N$.

\section{Proof of Conjecture~\ref{conj: provoke}}\label{sec:edgeless}

We prove Theorem~\ref{thm:edgeless} which implies
Conjecture~\ref{conj: provoke} raised in~\cite{KM25}. Without loss of
generality, we assume that $C_1, \ldots, C_n$ are $k$-element subsets of $C =
[k n]$. As a first step, we express the problem in terms of a suitable
generating polynomial. Given $\boldsymbol{C}= (C_1, \ldots, C_n)$ as in
Theorem~\ref{thm:edgeless}, we define
\begin{equation*}
  N_{\boldsymbol{C}} = \left| \left\{ [c_1, \ldots, c_n] \; : \; c_j \in C_j
   \right\} \right| .
\end{equation*}
We further introduce
\begin{equation*}
  P_{\boldsymbol{C}} (\boldsymbol{x}) = \prod_{i = 1}^n \sum_{c \in C_i} x_c
\end{equation*}
where $\boldsymbol{x}$ is the collection of variables $(x_c)_{c \in C}$. We
assume that our variables commute so that the following useful observation
follows by construction.

\begin{proposition}
  $N_{\boldsymbol{C}}$ equals the number of monomials in $P_{\boldsymbol{C}}
  (\boldsymbol{x})$.
\end{proposition}

\begin{example}
  Let $n = k = 3$. In the case $C_1 = C_2 = C_3 = \{ 1, 2, 3 \}$, we have
  \begin{eqnarray*}
    P_{\boldsymbol{C}} (\boldsymbol{x}) & = & (x_1 + x_2 + x_3)^3\\
    & = & x_1^3 + x_2^3 + x_3^3 + 6 x_1 x_2 x_3\\
    &  & + 3 x_1^2 x_2 + 3 x_1 x_2^2 + 3 x_1^2 x_3 + 3 x_1 x_3^2 + 3 x_2^2
    x_3 + 3 x_2 x_3^2 .
  \end{eqnarray*}
  The number of monomials is $N_{\boldsymbol{C}} = 10 = \binom{3 + 3 - 1}{3}$.  Observe that each monomial represents a multiset.
\end{example}
Given $c, d \in C$, we now introduce an operator $\mathcal{R}_{c, d}$ on
tuples $\boldsymbol{C}= (C_1, \ldots, C_n)$ that replaces the color $c$ by $d$
``where possible''. More precisely, $\mathcal{R}_{c, d} (\boldsymbol{C})
=\boldsymbol{D}$ with $\boldsymbol{D}= (D_1, \ldots, D_n)$ where
\begin{equation*}
  D_i = \left\{\begin{array}{ll}
     \{ d \} \cup C_i \backslash \{ c \}, & \text{if $c \in C_i$ and $d \not\in
     C_i$},\\
     C_i, & \text{otherwise} .
   \end{array}\right.
\end{equation*}
We note that this kind of operation, often referred to as shifting, is a
common ingredient in proofs in combinatorics and, in particular, extremal set
theory \cite{F87}.

\begin{lemma}
  \label{lem:NCD}If $\mathcal{R}_{c, d} (\boldsymbol{C}) =\boldsymbol{D}$ then
  $N_{\boldsymbol{C}} \geq N_{\boldsymbol{D}}$.
\end{lemma}

\begin{proof}
  For each $i \in [n]$ write
  \begin{equation*}
    \sum_{c \in C_i} x_c = p_i (\boldsymbol{x}) + \delta_i x_c + \varepsilon_i
     x_d
  \end{equation*}
  where $\delta_i, \varepsilon_i \in \{ 0, 1 \}$ and where $p_i
  (\boldsymbol{x})$ does not involve the variables $x_c$ and $x_d$. Write $I
  (\delta, \varepsilon)$ for the set of indices $i$ such that $\delta_i =
  \delta$ and $\varepsilon_i = \varepsilon$. By construction, we have
  \begin{equation*}
    P_{\boldsymbol{C}} (\boldsymbol{x}) = \prod_{i \in I (0, 0)} p_i \prod_{i \in
     I (1, 0)} (p_i + x_c) \prod_{i \in I (0, 1)} (p_i + x_d) \prod_{i \in I
     (1, 1)} (p_i + x_c + x_d) 
  \end{equation*}
  where any empty products are taken to be 1.  Upon partial expansion (leaving the terms $x_c + x_d$ from the last product
  combined), this can be expressed as
  \begin{equation*}
    P_{\boldsymbol{C}} (\boldsymbol{x}) = \sum_M M (\boldsymbol{x}) \sum_{(r, s, t)
     \in J_M} \lambda_{r, s, t}^{(M)} x_c^r x_d^s (x_c + x_d)^t
  \end{equation*}
  where the $M (\boldsymbol{x})$ are distinct monomials that do not involve the
  variables $x_c$ and $x_d$ and the $\lambda_{r, s, t}^{(M)}$ are positive
  integers. Crucially, since $\boldsymbol{D}=\mathcal{R}_{c, d} (\boldsymbol{C})$
  and since $\mathcal{R}_{c, d}$ only changes those $C_i$ for which $i \in I
  (1, 0)$, we also have the corresponding expansion
  \begin{equation*}
    P_{\boldsymbol{D}} (\boldsymbol{x}) = \sum_M M (\boldsymbol{x}) \sum_{(r, s, t)
     \in J_M} \lambda_{r, s, t}^{(M)} x_d^{r + s} (x_c + x_d)^t .
  \end{equation*}
  Observe that the claim $N_{\boldsymbol{C}} \geq N_{\boldsymbol{D}}$ follows
  if we can show that the number of monomials in $q_{\boldsymbol{C}}$ exceeds
  the number of monomials in $q_{\boldsymbol{D}}$ where
  \begin{equation*}
    q_{\boldsymbol{C}} = \sum_{(r, s, t) \in J_M} \lambda_{r, s, t}^{(M)} x_c^r
     x_d^s (x_c + x_d)^t, \quad q_{\boldsymbol{D}} = \sum_{(r, s, t) \in J_M}
     \lambda_{r, s, t}^{(M)} x_d^{r + s} (x_c + x_d)^t .
  \end{equation*}
  Note that $q_{\boldsymbol{C}}$ and $q_{\boldsymbol{D}}$ are polynomials in $x_c,
  x_d$ that are homogeneous of degree $n - \deg (M)$. When counting the
  monomials in $q_{\boldsymbol{C}}$ and $q_{\boldsymbol{D}}$ we can therefore
  focus on only the possible powers of $x_c$ that occur. For
  $q_{\boldsymbol{C}}$ and $q_{\boldsymbol{D}}$, these are those with exponent in
  \begin{equation*}
    \bigcup_{(r, s, t) \in J_M} \{ r, r + 1, \ldots, r + t \}, \quad
     \bigcup_{(r, s, t) \in J_M} \{ 0, 1, \ldots, t \},
  \end{equation*}
  respectively (so that the number of monomials equals the size of these
  sets). Let $t_{\max}$ be the maximal value of $t$ among the triples $(r, s,
  t) \in J_M$. It follows that $q_{\boldsymbol{D}}$ has exactly $t_{\max} + 1$
  monomials. On the other hand, we readily see that $q_{\boldsymbol{C}}$ has at
  least $t_{\max} + 1$ monomials coming just from a single triple of the form
  $(r, s, t_{\max}) \in J_M$.
\end{proof}

We are now in a convenient position to prove Theorem~\ref{thm:edgeless} which
we restate in equivalent form below.

\begin{theorem}
  \label{thm:edgeless:pf}If $C_1, \ldots, C_n$ are sets of size $k$, then
  \begin{equation}
    \left| \left\{ [c_1, \ldots, c_n] \; : \; c_j \in C_j \right\} \right|
    \geq \left| \left\{ [c_1, \ldots, c_n] \; : \; c_j \in [k] \right\}
    \right| . \label{eq:edgeless:pf}
  \end{equation}
\end{theorem}

\begin{proof}
  Set $\boldsymbol{C}= (C_1, \ldots, C_n)$ and assume, without loss of
  generality, that $C_j \subseteq [k n]$. We claim that repeated application
  of the operator $\mathcal{R}_{c, d}$ for suitable choices of $c$ and $d$
  will eventually result in $\boldsymbol{E}= (E, \ldots, E)$ with $E = [k]$. By
  Lemma~\ref{lem:NCD}, this implies that $N_{\boldsymbol{C}} \geq
  N_{\boldsymbol{E}}$ which is \eqref{eq:edgeless:pf}.
  
  To see that it is always possible to reach $\boldsymbol{E}$ in this fashion
  from any initial $\boldsymbol{C}$, consider the quantity
  \begin{equation*}
    w (\boldsymbol{C}) = \sum_{j = 1}^n \sum_{c \in C_j} c
  \end{equation*}
  which we use to measure the ``weight'' of $\boldsymbol{C}$ (recall that we are
  assuming $C_j \subseteq [k n]$). Clearly, $\boldsymbol{E}$ has the minimal
  possible weight. For any $\boldsymbol{C}$ besides $\boldsymbol{E}$ there exists
  an index $j$ such that $C_j$ contains an element $c > k$. In that case,
  there must also be an element $d \in [k]$ such that $d \not\in C_j$. The
  operator $\mathcal{R}_{c, d}$ applied to $\boldsymbol{C}$ therefore alters
  $C_j$. Write $\boldsymbol{D}=\mathcal{R}_{c, d} (\boldsymbol{C})$. By
  construction, since $d < c$, it is clear that $w (\boldsymbol{D}) < w
  (\boldsymbol{C})$. If this process is repeated, the weights strictly decrease
  until we reach the minimal-weight assignment $\boldsymbol{E}$ in a finite
  number of steps.
\end{proof}

\begin{example}
  In the case $n = 3, k = 4$, take $\boldsymbol{C}_0 = (C_1, C_2, C_3)$ with
  \begin{equation*}
    C_1 = \{ 1, 2, 3, 4 \}, \quad C_2 = \{ 4, 5, 6, 7 \}, \quad C_3 = \{ 1,
     2, 4, 5 \} .
  \end{equation*}
  We apply $\mathcal{R}_{7, 1}$, followed by $\mathcal{R}_{6, 2}$,
  $\mathcal{R}_{5, 3}$ to obtain $\boldsymbol{C}_1, \boldsymbol{C}_2,
  \boldsymbol{C}_3$ which together with $N_{\boldsymbol{C}_i}$ are recorded in the
  following table (note that $N_{\boldsymbol{C}_3} = \binom{3 + 4 - 1}{3} =
  20$):
  \begin{equation*}
    \begin{array}{|l|l|l|}
       \hline
       \boldsymbol{C} & (C_1, C_2, C_3) & N_{\boldsymbol{C}}\\
       \hline
       \boldsymbol{C}_0 & \{ 1, 2, 3, 4 \}, \{ 4, 5, 6, 7 \}, \{ 1, 2, 4, 5 \} &
       48\\
       \hline
       \boldsymbol{C}_1 & \{ 1, 2, 3, 4 \}, \{ 1, 4, 5, 6 \}, \{ 1, 2, 4, 5 \} &
       40\\
       \hline
       \boldsymbol{C}_2 & \{ 1, 2, 3, 4 \}, \{ 1, 2, 4, 5 \}, \{ 1, 2, 4, 5 \} &
       29\\
       \hline
       \boldsymbol{C}_3 & \{ 1, 2, 3, 4 \}, \{ 1, 2, 3, 4 \}, \{ 1, 2, 3, 4 \} &
       20\\
       \hline
     \end{array}
  \end{equation*}
\end{example}

\begin{example}
  In the case $n = 4, k = 3$, take $\boldsymbol{C}_0 = (C_1, C_2, C_3, C_4)$
  with
  \begin{equation*}
    C_1 = \{ 1, 2, 3 \}, \quad C_2 = \{ 2, 3, 4 \}, \quad C_3 = \{ 3, 4, 5
     \}, \quad C_4 = \{ 1, 2, 5 \} .
  \end{equation*}
  We apply $\mathcal{R}_{4, 1}$, $\mathcal{R}_{5, 2}$, $\mathcal{R}_{5, 3}$ to
  obtain $\boldsymbol{C}_1, \boldsymbol{C}_2, \boldsymbol{C}_3$ which together with
  $N_{\boldsymbol{C}_i}$ are recorded in the following table (note that
  $N_{\boldsymbol{C}_3} = \binom{4 + 3 - 1}{4} = 15$):
  \begin{equation*}
    \begin{array}{|l|l|l|}
       \hline
       \boldsymbol{C} & (C_1, C_2, C_3, C_4) & N_{\boldsymbol{C}}\\
       \hline
       \boldsymbol{C}_0 & \{ 1, 2, 3 \}, \{ 2, 3, 4 \}, \{ 3, 4, 5 \}, \{ 1, 2,
       5 \} & 47\\
       \hline
       \boldsymbol{C}_1 & \{ 1, 2, 3 \}, \{ 1, 2, 3 \}, \{ 1, 3, 5 \}, \{ 1, 2,
       5 \} & 29\\
       \hline
       \boldsymbol{C}_2 & \{ 1, 2, 3 \}, \{ 1, 2, 3 \}, \{ 1, 2, 3 \}, \{ 1, 2,
       5 \} & 24\\
       \hline
       \boldsymbol{C}_3 & \{ 1, 2, 3 \}, \{ 1, 2, 3 \}, \{ 1, 2, 3 \}, \{ 1, 2,
       3 \} & 15\\
       \hline
     \end{array}
  \end{equation*}
\end{example}

\section{Discussion}

We have proved Conjecture~\ref{conj: provoke} posed in \cite{KM25}, thus showing that Question~\ref{ques: fundamental} has an affirmative answer for edgeless graphs. Despite its deceptively simple appearance, this result implies, as observed in Corollary~\ref{cor:multi}, that if every connected component $H$ of a graph $G$ has the property that $P_{\ell}^u (H, k) = P^u (H, k)$ whenever $k \geq N$, then the graph $G$ itself also satisfies $P_{\ell}^u (G, k) = P^u (G, k)$ whenever $k \geq N$. In particular, a graph $G$ satisfies the statement in Question~\ref{ques: fundamental} if each of its connected components does so.

Despite its simplicity, we have not been able to locate Theorem~\ref{thm:edgeless} in the literature. Given that our proof in Section~\ref{sec:edgeless}, like many results in extremal set theory~\cite{F87}, uses shifting as a crucial ingredient, it is natural to wonder whether one can conclude Theorem~\ref{thm:edgeless} from a known more general result.

Next, in the appendix, we give a cleaned-up version of the AI-produced proof of Theorem~\ref{thm:edgeless}. The key idea is in Lemma~\ref{AI-lem1} which proves a shadow inequality for a family of multisets in the style of the probabilistic proof of the local LYM inequality (see~\cite{B86}) by Frankl~\cite{F83}. In some sense, our approach (shifting arguments) and the AI's approach (shadow inequalities) are the two natural classic approaches to problems in set systems like Theorem~\ref{thm:edgeless}.

Note that the main question from~\cite{KM25} remains open. Gowers~\cite{G26} asked ChatGPT 5.5 Pro to prove/disprove Question~\ref{ques: fundamental} for all graphs but it was unable to do so. Currently, we do not know of any approach for tackling this question in full generality.

\appendix\section{An alternative AI-produced proof}\label{sec:edgeless:ai}

In Section~\ref{sec:edgeless}, we proved Theorem~\ref{thm:edgeless} using a generating function argument. In this appendix, we offer an alternative proof that is extracted from the AI-generated output that was obtained by Gowers~\cite{G26} using ChatGPT 5.5 Pro and shared with the first two authors.

We use the standard notation that, for a set $B$, $\multichoose{B}{s}$ is the set of all multisets of size $s$ with elements from $B$.

\begin{lemma}\label{AI-lem1}
  \label{lem: step1}Suppose $B$ is a $q$-element set and $\mathcal{A}
  \subseteq \multichoose{B}{s}$. Let $\mathcal{A}_B^+ = \{A \cup \{b\}: A \in
  \mathcal{A}, b \in B\}$. Then, for any $s \geq 0$,
  \begin{equation*}
    \frac{|\mathcal{A}_B^+ |}{\binom{q + s}{s + 1}} \geq
     \frac{|\mathcal{A}|}{\binom{q + s - 1}{s}} .
  \end{equation*}
\end{lemma}

\begin{proof}
  We may suppose $B = [q]$. We may also view each element, $M$, of
  $\multichoose{B}{s}$ as a $q$-tuple, $(x_1, \ldots, x_q)$, of nonnegative
  integers satisfying $x_1 + \ldots + x_q = s$ where $x_i$ indicates the
  number of copies of $i$ in $M$.
  
  We now describe a procedure for choosing an element, $Y$, of
  $\multichoose{B}{s + 1}$. First, choose a $Z = (z_1, \ldots, z_q)$ from
  $\multichoose{B}{s}$ uniformly at random. Then, to obtain $Y$, choose a
  coordinate of $Z$ to increase by 1 so that the probability that the
  $i^{th}$-coordinate of $Z$ increases by 1 is $\frac{z_i + 1}{s + q}$. Note
  that for any $(y_1, \ldots, y_q) \in \multichoose{B}{s + 1}$,
  \begin{equation*}
    P (Y = (y_1, \ldots, y_q)) = \sum_{i = 1}^q \frac{1}{\binom{q + s -
     1}{s}}  \frac{y_i}{s + q} = \frac{s + 1}{\binom{q + s - 1}{s} (s + q)} =
     \frac{1}{\binom{q + s}{s + 1}} .
  \end{equation*}
  Consequently, the procedure for choosing $Y$ is the same as choosing $Y$
  uniformly at random.
  
  So, if $Z$ is chosen from $\multichoose{B}{s}$ uniformly at random, and $Y$
  is chosen via the coupling procedure above, the probability $Y \in
  \mathcal{A}_B^+$ is at least the probability $Z \in \mathcal{A}$. The result
  immediately follows.
\end{proof}

\begin{lemma}
  \label{lem: step2}Suppose $\mathcal{A}$ is a set of multisets each of size
  $r$ where each element of $\mathcal{A}$ has its elements in some universal
  set $U$. Let $B$ be a set of size $q$ with $B \subseteq U$. If
  \begin{equation*}
    \mathcal{A}+ B = \{A \cup \{b\}: A \in \mathcal{A}, b \in B\},
  \end{equation*}
  then
  \begin{equation*}
    |\mathcal{A}+ B| \geq \frac{\binom{q + r}{r + 1}}{\binom{q + r -
     1}{r}} |\mathcal{A}| .
  \end{equation*}
\end{lemma}

\begin{proof}
  Let $C = U - B$. For each multiset $T$ with elements in $C$ and $|T|
  \leq r$, let
  \begin{equation*}
    \mathcal{A} (T) = \left\{ S \in \multichoose{B}{r - |T|} : S \cup T \in
     \mathcal{A} \right\} .
  \end{equation*}
  Since each element of $\mathcal{A}$ can be uniquely written as the disjoint
  union of a multiset with elements from $B$ and a multiset with elements from
  $C$, $|\mathcal{A}| = \sum_T |\mathcal{A}(T) |$. Moreover,
  \begin{equation*}
    |\mathcal{A}+ B| = \sum_T |\mathcal{A}(T)_B^+ | .
  \end{equation*}
  When $s = r - |T|$, Lemma~\ref{lem: step1} gives
  \begin{equation*}
    |\mathcal{A}(T)_B^+ | \geq \frac{\binom{q + s}{s + 1}}{\binom{q + s
     - 1}{s}} |\mathcal{A}(T) | = \frac{q + s}{s + 1} |\mathcal{A}(T) |
     \geq \frac{q + r}{r + 1} |\mathcal{A}(T) | = \frac{\binom{q + r}{r +
     1}}{\binom{q + r - 1}{r}} |\mathcal{A}(T) | .
  \end{equation*}
  Summing over all $T$ gives the result.
\end{proof}

We are now ready to prove Theorem~\ref{thm:edgeless}, restated below for convenience.

\begin{theorem}
  Suppose that $C_1, \ldots, C_n$ are sets of size $k$. Then, the set of
  multisets
  \begin{equation*}
    \left\{ [c_1, \ldots, c_n] \; : \; c_j \in C_j \right\}
  \end{equation*}
  has size at least $\binom{k + n - 1}{n}$.
\end{theorem}

\begin{proof}
  Suppose $k \in \N$ is fixed. The proof is by induction on $n$. First note
  that the result is clear when $n = 1$. So, we may suppose $n \geq 2$
  and the desired result holds for all positive integers less than $n$. Let
  $\mathcal{A}= \{ [c_1, \ldots, c_{n - 1}] : c_j \in C_j \}$. By the
  induction hypothesis, $|\mathcal{A}| \geq \binom{k + n - 2}{n - 1}$.
  Notice that $\mathcal{A}+ C_n = \{ [c_1, \ldots, c_n] : c_j \in C_j
  \}$. By Lemma~\ref{lem: step2},
  \begin{equation*}
    |\mathcal{A}+ C_n | \geq \frac{\binom{k + n - 1}{n}}{\binom{k + n -
     2}{n - 1}} |\mathcal{A}| \geq \binom{k + n - 1}{n} .
  \end{equation*}
\end{proof}

\begin{acknowledgements}
We are grateful to W.~T.~Gowers for sharing with the first two authors the AI-generated output described above and for allowing us to include a paraphrased version of one of the resulting proofs as an appendix. 
\end{acknowledgements}

\bibliographystyle{hplain}
\bibliography{bibliography}

\end{document}